\documentclass{amsart}
\usepackage{amssymb,amsmath,amsthm,enumerate,amscd,graphicx,color}
\usepackage{pdfsync}
 \usepackage[colorlinks=true, pdfstartview=FitV, linkcolor=blue, citecolor=blue, urlcolor=blue]{hyperref}
\numberwithin{equation}{section}
\newtheorem{thm}{Theorem}[subsection]

\newtheorem{cor}[thm]{Corollary}
\newtheorem{prop}[thm]{Proposition}
\newtheorem{rem}[thm]{Remark}

\newcommand{\propref}[1]{Proposition~\ref{#1}}

\newcommand{\eqnref}[1]{~(\ref{#1})}

\begin{document}



\title{Imaginary Verma  Modules for $U_q(\widehat{\mathfrak{sl}(2)})$ and Crystal-like bases}
\author{ Ben Cox}
\author{Vyacheslav Futorny}
\author{Kailash C. Misra}
\address{Department of Mathematics \\
University of Charleston \\
66 George Street  \\
Charleston SC 29424, USA}\email{coxbl@cofc.edu}
\address{Department of Mathematics\\
 University of S\~ao Paulo\\
 S\~ao Paulo, Brazil}
 \email{futorny@ime.usp.br}
 \address{Department of Mathematics\\
 North Carolina State University\\
 Raleigh, NC 27695-8205, USA}
 \begin{abstract}  We consider imaginary Verma modules for quantum affine algebra
 $U_q(\widehat{\mathfrak{sl}(2)})$  and define a crystal-like base which we call an imaginary crystal basis 
 using  the Kashiwara algebra $\mathcal K_q$ constructed in earlier work of the authors. In particular, we prove the existence of imaginary like bases 
 for a suitable category of reduced imaginary Verma modules for $U_q(\widehat{\mathfrak{sl}(2)})$.
\end{abstract}
\date{}
\thanks{}


\maketitle
\section{Introduction} 

 We consider imaginary Verma modules for quantum affine algebra
 $U_q(\widehat{\mathfrak{sl}(2)})$  and define a crystal-like base which we call an imaginary crystal basis 
 using  the Kashiwara algebra $\mathcal K_q$ constructed in earlier work of the authors. In particular, we prove the existence of imaginary crystal-like bases 
 for a suitable category of reduced imaginary Verma modules for $U_q(\widehat{\mathfrak{sl}(2)})$.

Consider the affine Lie algebra  $\widehat{\mathfrak{g}} = \widehat{\mathfrak{sl}(2)}$  with Cartan subalgebra $\widehat{\mathfrak{h}}$.
Let $\{\alpha_0 , \alpha_1\}$ be the simple roots, $\delta = \alpha_0 + \alpha_1$ the null root and $\Delta$ the set of roots for 
$\widehat{\mathfrak{g}}$ with respect to $\widehat{\mathfrak{h}}$. Then we have a natural (standard) partition of $\Delta = \Delta_+ \cup \Delta_-$ into set of positive and negative roots. Corresponding to this standard partition we have a standard Borel subalgebra from which we induce the standard Verma module. Let $S =  \{ \alpha_1+k\delta \ |\ k\in \mathbb Z \} \cup \{l\delta\ |\ l \in \mathbb Z_{>0} \}$. Then $\Delta = S \cup -S$ is another closed partition of the root system $\Delta$ which is not Weyl group conjugate to the standard partition. The classification of closed partitions of the root system for affine Lie algebras was obtained by Jakobsen and Kac \cite{JK,MR89m:17032}, and independently by Futorny \cite{MR1078876,MR1175820}. In fact for affine Lie algebras there exists a finite number ($\geq 2$) of inequivalent Weyl group orbits of closed partitions. For the affine Lie algebra 
$\widehat{\mathfrak{g}}$ the partition $\Delta = S \cup -S$ is the only nonstandard closed partition which gives rise to a nonstandard Borel subalgebra. The Verma module $M(\lambda)$ with highest weight $\lambda$ induced by this nonstandard Borel subalgebra is called the imaginary Verma module for $\widehat{\mathfrak{g}}$. Unlike the standard Verma module, the imaginary Verma module $M(\lambda)$ contain both finite and infinite dimensional weight spaces.
 

For generic $q$, consider the associated quantum affine algebra $U_q(\widehat{\mathfrak{g}})$  (\cite{MR802128}, \cite{MR797001}). Lusztig \cite{MR954661} proved that the integrable highest weight modules of  $\widehat{\mathfrak{g}}$ can be deformed to those over $U_q(\widehat{\mathfrak{g}})$ in such a way that the dimensions of the weight spaces are invariant under the deformation. 
Following the framework of \cite{MR954661} and \cite{MR1341758}, it was shown in  (\cite {MR97k:17014}, \cite{MR1662112}) that the imaginary Verma modules $M(\lambda)$ can also be $q$-deformed to the quantum imaginary Verma modules $M_q(\lambda)$ in such a way that the weight multiplicities, both finite and infinite-dimensional, are preserved.


Lusztig \cite{MR1035415} from a geometric view point and Kashiwara \cite{MR1115118} from an algebraic view point 
introduced the notion of canonical bases (equivalently, global crystal bases) for standard Verma modules
$V_q(\lambda)$ and integrable highest weight modules $L_q(\lambda)$. The crystal base (\cite{MR1090425, MR1115118}) can be thought of as the $q=0$ limit of the global crystal base or canonical base. An important ingredient in the construction of crystal base by Kashiwara in
 \cite{MR1115118}, is a subalgebra $\mathcal {B}_q$ of the quantum group which acts on the negative part of the quantum group  
 by left multiplication. This subalgebra $\mathcal {B}_q$, which we call the Kashiwara algebra, played an important role in the definition of the Kashiwara operators which defines the crystal base. In \cite{CFM10} we constructed an analog of Kashiwara algebra, denoted by $\mathcal K_q$ for the imaginary Verma module $M_q(\lambda)$ for the quantum affine algebra $U_q(\widehat{\mathfrak{g}})$ by introducing certain Kashiwara-type operators. Then we proved that a certain quotient $\mathcal N_q^-$ of  $U_q(\widehat{\mathfrak{g}})$ is a simple $\mathcal K_q$-module and gave a necessary and sufficient condition for a particular quotient  $\tilde{M}_q(\lambda)$ (called reduced imaginary Verma module) of  $M_q(\lambda)$ to be simple. These results were generalized to any affine Lie algebra of $ADE$ type in \cite{CFM14}.


In this paper we consider a category $\mathcal O^q_{\text{red,im}}$ of $U_q(\widehat{\mathfrak{g}})$-modules and define a crystal-like basis which we call imaginary crystal basis for modules in this category. We show that the reduced imaginary Verma modules $\tilde{M}_q(\lambda)$ are in $\mathcal O^q_{\text{red,im}}$. Then we show that any module in $\mathcal O^q_{\text{red,im}}$ is a direct sum of reduced imaginary Verma modules for $U_q(\widehat{\mathfrak{g}})$. Finally we prove the existence of imaginary crystal basis for the reduced imaginary Verma module $\tilde{M}_q(\lambda)$.

The paper is organized as follows.
In Sections 2 we recall necessary definitions and properties about the algebra $U_q(\widehat{\mathfrak{g}})$ that we need. In Section 3 we recall the definitions and relations of $\Omega$-operators defined in  \cite{CFM10}. In Section 4, we recall the definition of the Kashiwara algebra $\mathcal K_q$ and the symmetric bilinear form  $( \, , \, )$ on the simple $\mathcal K_q$-module $\mathcal N_q^-$ from \cite{CFM10} and show that this form satisfies certain orthonormality condition modulo $q^2$ and is non-degenerate. In Section 5 we recall the definitions and properties of imaginary Verma modules $M(\lambda)$ for the affine Lie algebra $\widehat{\mathfrak{g}}$ and the reduced imaginary Verma modules $\tilde{M}(\lambda)$. In Section 6 we define the category $\mathcal O_{\text{red,im}}$ of $\widehat{\mathfrak{g}}$-modules and show that this category is a Serre category and any module in this category is a direct sum of some simple reduced imaginary Verma modules. In Section 7 we recall some basic results about quantized imaginary Verma modules and reduced quantized imaginary Verma modules for $U_q(\widehat{\mathfrak{g}})$. In Section 8 we define the category $\mathcal O^q_{\text{red,im}}$ of $U_q(\widehat{\mathfrak{g}})$-modules containing the reduced quantized imaginary modules  $\tilde{M}_q(\lambda)$ and define Kashiwara type operators $\tilde\Omega_\psi(m)$ and $\tilde x_m^-$ on $\tilde M_q(\lambda)$. In Section 9 we define the imaginary crystal basis for any module $M \in \mathcal O^q_{\text{red,im}}$ and prove the existence of an imaginary crystal basis for any reduced quantized Verma module $\tilde M_q(\lambda)$.

 \section{Notation}
 
\subsection{}  Let $\mathbb{F}$ denote a field of characteristic zero. The {\it quantum group}
$U_q(A_1^{(1)})$ is the $\mathbb F(q^{1/2})$-algebra with 1 generated by
$$ e_0, e_1, f_0, f_1, K_0^{\pm 1}, K_1^{\pm 1}, D^{\pm 1} $$
with defining relations:
\begin{align*}& DD^{-1}=D^{-1}D=K_iK_i^{-1}=K_i^{-1}K_i=1, \\
& e_if_j-f_je_i = \delta_{ij}\frac{K_i-K_i^{-1}}{q-q^{-1}}, \\
& K_ie_iK_i^{-1}=q^2e_i, \ \ K_i f_i K_i^{-1} =q^{-2}f_i, \\
& K_i e_jK_i^{-1} = q^{-2}e_j, \ \
K_i f_jK_i^{-1} = q^2f_j, \quad i\neq j, \\
& K_iK_j-K_jK_i = 0, \ \ K_iD-DK_i=0, \\
& De_iD^{-1}=q^{\delta_{i,0}} e_i, \ \
Df_iD^{-1}=q^{-\delta_{i,0}} f_i, \\
& e_i^3e_j-[3]e_i^2e_je_i+[3]e_ie_je_i^2-e_je_i^3 =0, \quad i\neq j, \\
& f_i^3f_j-[3]f_i^2f_jf_i+[3]f_if_jf_i^2-f_jf_i^3 = 0, \quad i\neq j, \\
\end{align*}where, $[n] = \frac{q^n-q^{-n}}{q-q^{-1}}$.

The quantum group $U_q(A_1^{(1)})$
can be given a Hopf algebra structure with a comultiplication given by
\begin{align*}
\Delta(K_i) &= K_i \otimes K_i, \\
\Delta(D)&=D\otimes D, \\
\Delta(e_i) &= e_i\otimes K_i^{-1} + 1\otimes e_i, \\
\Delta(f_i) &= f_i\otimes 1 + K_i \otimes f_i, \\
\end{align*}and an antipode given by
\begin{align*}
s(e_i) &=-e_iK_i^{-1}, \\
s(f_i) &= -K_if_i, \\
s(K_i) &= K_i^{-1},\\
s(D) &= D^{-1}.
\end{align*}
 There is an alternative realization for $U_q(A_1^{(1)})$,
due to Drinfeld
\cite{MR802128}, which we shall also need.  Let
$U_q$ be the associative algebra with $1$ over $\mathbb F(q^{1/2})$
generated by the
elements $x^{\pm }_k$ ($k\in \mathbb Z$), $h_l$ ($l \in \mathbb Z
\setminus \{0\}$), $K^{\pm 1}$,
$D^{\pm 1}$, and $\gamma^{\pm \frac12}$ with the following defining
relations:
\begin{align}
DD^{-1}=D^{-1}D&=KK^{-1}=K^{-1}K=1,  \\
[\gamma^{\pm \frac 12},u] &= 0 \quad \forall u \in U, \\
[h_k,h_l] &= \delta_{k+l,0} \frac{[2k]}{k} \frac{\gamma^k -
\gamma^{-k}}{q-q^{-1}},  \\
[h_k,K]&=0,\quad [D,K]=0,  \\
Dh_kD^{-1}&=q^k h_k,  \\
Dx^{\pm}_kD^{-1}&=q^{ k}x^{\pm}_k,  \\
Kx^{\pm}_kK^{-1} &= q^{\pm 2}x^{\pm}_k,   \\
[h_k,x^{\pm}_l]&= \pm \frac{[2k]}{k}\gamma^{\mp \frac{|k|}{2}}x^{\pm}_{k+l}, \label{axcommutator}  \\
    x^{\pm}_{k+1}x^{\pm}_l &- q^{\pm 2}
x^{\pm}_lx^{\pm}_{k+1}\label{Serre}   \\
&= q^{\pm 2}x^{\pm}_kx^{\pm}_{l+1}
    - x^{\pm}_{l+1}x^{\pm}_k,\notag \\
[x^+_k,x^-_l]&=
    \frac{1}{q-q^{-1}}\left( \gamma^{\frac{k-l}{2}}\psi(k+l) -
    \gamma^{\frac{l-k}{2}}\phi(k+l)\right), \label{xcommutator}   \\
\text{where  }
\sum_{k=0}^{\infty}\psi(k)z^{-k} &= K \exp\left(
(q-q^{-1})\sum_{k=1}^{\infty} h_kz^{-k}\right),\\
\sum_{k=0}^{\infty}
\phi(-k)z^k&= K^{-1} \exp\left( - (q-q^{-1})\sum_{k=1}^{\infty}
h_{-k}z^k\right).
\end{align}
The algebras $U_q(A_1^{(1)})$
and $U_q$ are isomorphic \cite{MR802128}. The action of the
isomorphism,
which we shall call the {\it Drinfeld Isomorphism,} on the generators of
$U_q(A_1^{(1)})$ is given by:
\begin{align*}e_0 &\mapsto x^-_1K^{-1}, \ \ f_0 \mapsto Kx^+_{-1}, \\
e_1 &\mapsto x^+_0, \ \ f_1 \mapsto x^-_0, \\
K_0 &\mapsto \gamma K^{-1}, \ \ K_1 \mapsto K, \ \ D \mapsto D.
\end{align*}
If one uses the formal sums
\begin{equation}
\phi(u)=\sum_{p\in\mathbb Z} \phi(p)u^{-p},\enspace \psi(u)=\sum_{p\in\mathbb Z}\psi(p)u^{-p},\enspace
x^{\pm }(u)=\sum_{p\in\mathbb Z} x^\pm(p)u^{-p}
\end{equation}
Drinfeld's relations (3), (8)-(10) can be written as
\begin{gather}
[\phi(u),\phi(v)]=0=[\psi(u),\psi(v)] \\
\phi(u)x^\pm (v)\phi(u)^{-1}=g(uv^{-1}\gamma^{\mp 1/2})^{\pm 1}x^\pm (v)\label{phix} \\
\psi(u)x^\pm (v)\psi(u)^{-1}=g(vu^{-1}\gamma^{\mp 1/2})^{\mp 1}x^\pm (v)\label{psix} \\
(u-q^{\pm2} v)x^\pm (u)x^\pm (v)=(q^{\pm 2}u-v)x^\pm(v)x^\pm(u) \\
[x^+(u),x^-(v)]=(q-q^{-1})^{-1}(\delta(u/v\gamma)\psi(v\gamma^{1/2})-\delta(u\gamma/v)\phi(u\gamma^{1/2}))\label{xx}
\end{gather}
where
$g(t)=g_q(t)=\sum_{k\geq 0}g(r)t^{k}$ is the Taylor series at $t=0$ of the function $(q^2t-1)/(t-q^2)$ and $\delta(z)=\sum_{k\in\mathbb Z}z^{k}$ is the formal Dirac delta function.
\begin{rem}
Writing $g(t)=g_q(t)=\sum_{r\geq 0}g(r)t^r$ we have
\begin{equation}\label{grcomp}
g(r)=g_q(r)=g_{q^{-1}}(r)=\begin{cases} q^{2}&\text{if}\quad r=0 \\ 
(1-q^{-4})q^{2(r+1)}=(q^{4}-1)q^{2(r-1)},&\text{if}\quad r>0.
\end{cases}
\end{equation}
\end{rem}

Considering Serre's relation  with $k=l$, we get
\begin{equation}\label{Serre1}
x^{-}_kx^{-}_{k+1}=q^2x^{-}_{k+1}x^{-}_k  
\end{equation}
The product on the right side is in the correct order for a basis  element.
If $k+1>l$ and $k\neq l$ in \eqnref{Serre}, then $k+1>  l+1$ so that $k\geq l+1$, and thus we can write
\begin{equation}\label{Serre2}
 x^{-}_lx^{-}_{k+1}=q^2x^{-}_{k+1}x^{-}_l   - x^{-}_kx^{-}_{l+1} +q^2 x^{-}_{l+1}x^{-}_k
\end{equation}
and then after repeating the above identity, we will eventually arrive at sums of terms that are in the correct order.    This is the opposite ordering of monomials as we had previously.

\section{$\Omega$-operators and their relations}

Let $\mathbb N^{\mathbb N^*}$ denote the set of all functions from $\{k\delta\,|\, k\in\mathbb N^*\}$ to $\mathbb N$ with finite support.   Then we can write
$$
h^+=h^{(s_k)}_+:=h_{r_1}^{s_1}\cdots h_{r_l}^{s_l},\quad h^-:=h^{(s_k)}_-=
h_{-r_1}^{s_1}\cdots h_{-r_l}^{s_l}
$$
for $f=(s_k)\in\mathbb N^{\mathbb N^*}$ whereby $f(r_k)=s_k$ and $f(t)=0$ for $t\neq r_i, 1\leq i\leq  l$.

Consider now the subalgebra $ \mathcal N_q^-$, generated by $\gamma^{\pm1/2}$, and $x^-_l$, $l\in\mathbb Z$. Note that the corresponding relations (9) hold in $ \mathcal N_q^-$.  Consider $x^-(v)=\sum_mx^-_mv^{-m}$ as a formal power series of left multiplication operators $x^-_m:\mathcal N_q^-\to \mathcal N_q^-$. 


As in our previous paper we set 
\begin{align*}
\bar P&=x^-(v_1)\cdots x^-(v_k) \\
\bar P_l&=x^-(v_{1})\cdots x^-(v_{l-1})x^-(v_{l+1})\cdots x^-(v_k),
\end{align*}
and
$$
G_l=G_l^{1/q}:=\prod_{j=1}^{l-1}g_{q^{-1}}(v_j/v_l),\quad G_l^{q}=\prod_{j=1}^{l-1}g(v_l /v_j)
$$
where $G_1:=1$.
As in our previous work we define a collection of operators $\Omega_\psi(k),\Omega_\phi(k):\mathcal N_q^-\to \mathcal N_q^-$, $k\in\mathbb Z$, in terms of the generating functions
$$
\Omega_\psi(u)=\sum_{l\in\mathbb Z}\Omega_\psi(l)u^{-l},\quad \Omega_\phi(u)=\sum_{l\in\mathbb Z}\Omega_\phi(l)u^{-l}
$$
by setting 
\begin{align}\label{definingomegapsi}
\Omega_\psi(u)(\bar P):&=\gamma^{m} \sum_{l=1}^kG_l
  \bar P_l \delta(u/v_l\gamma) \\
   \Omega_\phi(u)(\bar P):&=\gamma^{m}\sum_{l=1}^kG_l^q  \bar P_l\delta(u\gamma/v_l).\label{definingomegaphi}
\end{align}
%
%
%
Note that $\Omega_{\psi}(u)(1)=\Omega_{\phi}(u)(1)=0$.  More generally let us write
$$
\bar P=x^-(v_1)\cdots x^-(v_k)=\sum_{n\in\mathbb Z}\sum_{n_1,n_2,\dots,n_k\in\mathbb Z\atop n_1+\cdots +n_k=n}x^-_{n_1}\cdots x^-_{n_k}v_1^{-n_1}\cdots v_k^{-n_k}
$$
Then
\begin{align*}
\psi&(u\gamma^{-1/2})\Omega_\psi(u)(\bar P)\\
&=\sum_{k\geq 0}\sum_{p\in\mathbb Z}\sum_{n_i\in\mathbb Z } \gamma^{k/2}\psi(k)\Omega_\psi(p)(x^-_{n_1}\cdots x^-_{n_k})v_1^{-n_1}\cdots v_k^{-n_k}u^{-k-p}  \\
&=\sum_{n_i\in\mathbb Z } \sum_{m\in\mathbb Z}\sum_{k\geq 0}\gamma^{k/2}\psi(k)\Omega_\psi(m-k)(x^-_{n_1}\cdots x^-_{n_k})v_1^{-n_1}\cdots v_k^{-n_k}u^{-m}
\end{align*}
while
\begin{equation*}
[x^+(u),\bar P]=\sum_{m\in\mathbb Z}\sum_{n_1,n_2,\dots,n_k\in\mathbb Z } [x^+_m,x^-_{n_1}\cdots x^-_{n_k}]v_1^{-n_1}\cdots v_k^{-n_k}u^{-m}.
\end{equation*}
Thus for a fixed $m$ and $k$-tuple $(n_1,\dots,n_k)$ the sum
$$
\sum_{k\geq 0}\gamma^{k/2}\psi(k)\Omega_\psi(m-k)(x^-_{n_1}\cdots x^-_{n_k})
$$
must be finite.  Hence
\begin{equation}\label{omegalocalfin}
\Omega_\psi(m-k)(x^-_{n_1}\cdots x^-_{n_k})=0,
\end{equation}
 for $k$ sufficiently large.

\begin{prop} \label{commutatorprop}  Then
\begin{align}
\Omega_\psi(u)x^-(v)&=\delta(v\gamma/u)+g_{q^{-1}}(v\gamma/u)x^-(v)\Omega_\psi(u),
\label{omegapsi}\\
  \Omega_\phi(u)x^-(v)&=\delta(u\gamma/v)+g(u\gamma/v)x^-(v)\Omega_\phi(u)\label{omegaphi}  \\
(q^2u_1-u_2)\Omega_\psi(u_1)\Omega_\psi(u_2)&=(u_1-q^2u_2)\Omega_\psi(u_2)\Omega_\psi(u_1) \label{psipsi} \\
(q^2u_1-u_2)\Omega_\phi(u_1)\Omega_\phi(u_2)&=(u_1-q^2u_2)\Omega_\phi(u_2)\Omega_\phi(u_1)   \label{phiphi} \\
(q^2\gamma^2u_1-u_2)\Omega_\phi(u_1)\Omega_\psi(u_2)&=(\gamma^2u_1-q^2u_2)\Omega_\psi(u_2)\Omega_\phi(u_1)\label{omegaphipsi}
\end{align}
\end{prop}

The identities in \propref{commutatorprop} can be rewritten as
\begin{align}
(q^2v\gamma- u)\Omega_\psi(u)x^-(v)&=(q^2v\gamma- u)\delta(v\gamma/u)+(q^2v\gamma -u)x^-(v)\Omega_\psi(u),
\label{omegapsi2}\\
(q^2v- u\gamma)  \Omega_\phi(u)x^-(v)&=(q^2v- u\gamma)\delta(v/u\gamma)+( v-q^2u\gamma)x^-(v)\Omega_\phi(u)\label{omegaphi3}
\end{align}
which may be written out in terms of components as
\begin{align}
&q^2\gamma\Omega_\psi(m)x^-(n+1)- \Omega_\psi(m+1)x^-_n\label{omegapsi4} \\
&\quad =(q^2\gamma-1)\delta_{m,-n-1}+ \gamma x^-_{n+1}\Omega_\psi(m)-q^2x^-_n\Omega_\psi(m+1),
\notag \\
&  q^2\Omega_\phi(m)x^-(n+1)-  \gamma\Omega_\phi(m+1)x^-(n) \label{omegaphi5} \\
&\qquad =(q^2- \gamma)\delta_{m,-n-1}+ x^-(n+1)\Omega_\psi(m)-q^2\gamma x^-(n)\Omega_\psi(m+1).\notag 
\end{align}
%

  We also have by \eqnref{omegaphipsi}
  \begin{equation}\label{omegaphipsi2}
    \Omega_\psi(k)\Omega_\phi(m)= \sum_{r\geq 0}g (r)\gamma^{2r}\Omega_\phi(r+m)\Omega_\psi(k-r),
\end{equation}
as operators on $\mathcal N_q^-$.

  We can also write \eqnref{omegapsi} in terms of components and as operators on $\mathcal N_q^-$
\begin{equation}\label{omegapsi6}
    \Omega_\psi(k)x^-(m)=\delta_{k,-m}\gamma^{k}+\sum_{r\geq 0}g_{q^{-1}}(r)x^-(m+r)\Omega_\psi(k-r)\gamma^{r}.
\end{equation}
    The sum on the right hand side turns into a finite sum when applied to an element in $\mathcal N_q^-$, due to \eqnref{omegalocalfin}.

\section{The Kashiwara algebra $\mathcal K_q$}  

The Kashiwara algebra $\mathcal K_q$ is defined to be the $\mathbb F(q^{1/2})$-subalgebra of $\text{End}\,( \mathcal N_q)$ generated by $\Omega_\psi(m),x^-_n,\gamma^{\pm 1/2}$, $m,n\in\mathbb Z$, $\gamma^{\pm 1/2}$.  Then the $\gamma^{\pm 1/2}$ are central and the following relations (which are implied by \eqnref{omegapsi6}) are satisfied
\begin{align}
q^2\gamma\Omega_\psi(m)&x^-_{n+1}-  \Omega_\psi(m+1)x^-_n \\
&=(q^2\gamma-1)\delta_{m,-n-1}+ \gamma x^-_{n+1}\Omega_\psi(m)-q^2x^-_n\Omega_\psi(m+1) \notag \\\notag \\
q^2 \Omega_\psi(k+1)&\Omega_\psi(l) -
\Omega_\psi(l)\Omega_\psi(k+1)  =  \Omega_\psi(k)\Omega_\psi(l+1)
    - q^2\Omega_\psi(l+1)\Omega_\psi(k)\label{omegapsi3}
\end{align}
\begin{equation}
 x^{-}_lx^{-}_{k+1}-q^2x^{-}_{k+1}x^{-}_l  =q^2 x^{-}_{l+1}x^{-}_k - x^{-}_kx^{-}_{l+1} \label{xminusreln}
\end{equation}
 together with
\[
\gamma^{1/2}\gamma^{-1/2}=1=\gamma^{-1/2}\gamma^{1/2}.
\]

\begin{prop}\cite{CFM10} \label{form}
There is a unique symmetric bilinear form $(\enspace, \enspace)$ defined on $\mathcal N^-_q$ satisfying
$$
(x^-_ma,b)=(a,\Omega_\psi(-m)b),\quad (1,1)=1.
$$
\end{prop}

  For $\mathbf m=(m_1,\dots, m_n)$ set
$$
x_{\mathbf m}=x_{m_1}^-\cdots x_{m_n}^-
$$
and define the length of such a Poincare-Birkhoff-Witt basis element to be $|\mathbf m|=n$.
  \begin{prop}\label{symmetricform}  For  $\mathbf m=(m_1,\dots, m_n)\in\mathbb Z^n$, and $\mathbf k=(k_1,\dots, k_l)\in \mathbb Z^l$, if $n>l$, then
  \begin{equation} \label{orthonormal} 
 (x_{\mathbf m},x_{\mathbf k}) =0 . 
 \end{equation}
 On the other hand if $n=l$ with 
\begin{gather*}
  m_1\geq m_2\geq \cdots \geq m_n,\quad k_1\geq k_2\geq \cdots \geq k_n, \\
  \sum_{i=1}^nm_i=  \sum_{i=1}^nk_i
\end{gather*}
   we have
 \begin{equation} \label{orthonormal} 
 (x_{\mathbf m},x_{\mathbf k}) \equiv\delta_{\mathbf m,\mathbf k}\mod q^2\mathbb Z[\![q]\!].
\end{equation}
and the form is symmetric.
  \end{prop}

\begin{proof}   The fact that the form is symmetric comes from \propref{form} above.  Suppose $n>l$.  Then
\begin{align*} 
 (x_{m_1}&\cdots x_{m_n},x_{k_1}\cdots x_{k_l})= (x_{m_2}\cdots x_{m_n},\Omega_\psi(-m_1)x_{k_1}\cdots x_{k_l}) \\
 &=\delta_{m_1,k_1} (x_{m_2}\cdots x_{m_n}, x_{k_2}\cdots x_{k_n})   \\
 &\hskip 75pt+\sum_{r\geq 0}g_{q^{-1}}(r) (x_{m_2}\cdots x_{m_n},x_{k_1+r}\Omega_\psi(-m_1-r)x_{k_2}\cdots x_{k_l}).
  \end{align*}
By the Serre relations \eqnref{Serre1} and \eqnref{Serre2}
$$
x_{k_1+r}\Omega_\psi(-m_1-r)x_{k_2}\cdots x_{k_l}
$$
is a sum of monomials of length $l-1$ we can use induction to see that 
$$
(x_{m_2}\cdots x_{m_n},x_{k_1+r}\Omega_\psi(-m_1-r)x_{k_2}\cdots x_{k_l})=0.
$$
Hence $ (x_{m_1}\cdots x_{m_n},x_{k_1}\cdots x_{k_l})=0$.

Now suppose $n=l$.   For $n=1$ we have
\begin{equation*}  
 (x_{m},x_{k})=(1,\Omega_\psi(-m)x_k^-)=\delta_{m,k} \end{equation*}
by \eqnref{omegapsi6}.  

For $n=2$ we have by \eqnref{omegapsi6} for $m_1\geq m_2$, $k_1\geq k_2$ and $m_1+m_2=k_1+k_2$
\begin{align*} 
 (x_{m_1}x_{m_2},x_{k_1}x_{k_2})&= (x_{m_2},\Omega_\psi(-m_1)x_{k_1}x_{k_2}) \\
 &=\delta_{m_1,k_1} (x_{m_2}, x_{k_2})   +\sum_{r\geq 0}g_{q^{-1}}(r) (x_{m_2},x_{k_1+r}\Omega_\psi(-m_1-r)x_{k_2}) \\
 &=\delta_{\mathbf m,\mathbf k}     +\sum_{r\geq 0}g_{q^{-1}}(r) (x_{m_2},x_{k_1+r} )\delta_{m_1+r,k_2} \\
 &=\delta_{\mathbf m,\mathbf k}     +\sum_{r\geq 0}g_{q^{-1}}(r) \delta_{m_2,k_1+r}\delta_{m_1+r,k_2} \\
 &=\delta_{\mathbf m,\mathbf k}     +H(k_2-m_1)g_{q^{-1}}(k_2-m_1) \delta_{m_2-k_1,k_2-m_1 }  \\
  \end{align*}
where $H$ is the Heaviside function given by $H(n)=1$ if $n\geq 0$ and $H(n)=0$ otherwise.    Interchanging $(m_1,m_2)\leftrightarrow (k_1,k_2)$ in the above calculation we see that the $ (x_{m_1}x_{m_2},x_{k_1}x_{k_2})=(x_{l_1}x_{k_2},x_{m_1}x_{m_2})$.  Now if $k_2-m_1\neq 1$, then it is clear from \eqnref{grcomp}, that $ (x_{m_1}x_{m_2},x_{k_1}x_{k_2})\in \delta_{\mathbf m,\mathbf k}+q^2\mathbb Z[\![q]\!]$.  If $k_2-m_1=1$, then the second summand above is nonzero if and only if $m_2-k_1=1$.  But then
$$
m_1\geq m_2 =k_1+1>k_1\geq k_2=m_1+1
$$
which is impossible.  Hence for $n=2$, we have \eqnref{orthonormal}. 

Assume that \eqnref{orthonormal} holds up to Poincare-Birkhoff-Witt monomials of length $n-1$. 
Let us first prove by induction that for all $1\leq i\leq n-1$ and any $p\in\mathbb N$, 
\begin{align} 
(x^-_{m_2}\cdots x^-_{m_n},x^-_{s_1}x^-_{s_2}\cdots x^-_{s_i }\Omega_\psi(-m_1-p)x_{s_{i+1}}^-\cdots x_{s_n}^-)  \in \mathbb Z[\![q]\!] \label{omegaind},\\ 
(x^-_{m_2}\cdots x^-_{m_n},x^-_{s_2 }\cdots x_{s_n}^-)  \in \mathbb Z[\![q]\!], 
\end{align}
for any $\mathbf s=(s_2,\dots,s_n)\in\mathbb Z$ (so that $x^-_{s_2 }\cdots x_{s_n}^-$ is not necessarily a PBW monomial).  We say that $(s_2,\dots, s_n)$ has $k$ ascending inversions if the number of pairs of indices $(i,l)$ with $i<j$ and $s_i<s_j$ is $k$. Recall the Serre relations \eqnref{Serre1} and \eqnref{Serre2}.  Suppose there is an ascending inversion at the pair of indices $(i,i+1)$ with $s_i=k$ and $s_{i+1}=k+1$, then 
\begin{equation}
(x^-_{m_2}\cdots x^-_{m_n},x^-_{s_2 }\cdots x^-_{s_i } x_{s_{i+1}}^-\cdots x_{s_n}^-)= q^2
(x^-_{m_2}\cdots x^-_{m_n},x^-_{s_2 }\cdots x^-_{s_{i+1} } x_{s_{i}}^-\cdots x_{s_n}^-).
\end{equation}
 Then we have decreased the number of ascending inversions and by induction on the number of inversion on products of length $n-1$ we conclude
\begin{align*}
(x^-_{m_2}\cdots x^-_{m_n},x^-_{s_2 }\cdots x^-_{s_i } x_{s_{i+1}}^-\cdots x_{s_n}^-)\in  \mathbb Z[\![q]\!] .
\end{align*}
Suppose there is an ascending inversion at the pair of indices $(i,i+1)$ with $s_i=l$ and $s_{i+1}=k+1$ with $l<k$, then
 \begin{align*}
(x^-_{m_2}\cdots x^-_{m_n},x^-_{s_2 }\cdots x^-_{s_i } x_{s_{i+1}}^-\cdots x_{s_n}^-)&= (x^-_{m_2}\cdots x^-_{m_n},x^-_{s_2 }\cdots x^-_{l} x_{k+1}^-\cdots x_{s_n}^-)  \\
&= q^2(x^-_{m_2}\cdots x^-_{m_n},x^-_{s_2 }\cdots x^-_{k+1 } x_{l}^-\cdots x_{s_n}^-)  \\
&- (x^-_{m_2}\cdots x^-_{m_n},x^-_{s_2 }\cdots x^-_{k} x_{l+1}^-\cdots x_{s_n}^-)  \\
&+q^2 (x^-_{m_2}\cdots x^-_{m_n},x^-_{s_2 }\cdots x^-_{l+1 } x_{k}^-\cdots x_{s_n}^-).
\end{align*}
Observe that the number of ascending inversions in the first two summands has decreased by one and the last summand can also be rewritten as a sum of terms that have a decrease in the number of ascending inversions.  By induction on the number of inversion on products of length $n-1$ we again conclude
\begin{equation}\label{eqn1}
(x^-_{m_2}\cdots x^-_{m_n},x^-_{s_2 }\cdots x^-_{s_i } x_{s_{i+1}}^-\cdots x_{s_n}^-)\in  \mathbb Z[\![q]\!] .
\end{equation}

  For the first statement \eqnref{omegaind} we begin at $i=n-1$. By \eqnref{omegapsi6} and \eqnref{eqn1} this is 
\begin{align}
(x^-_{m_2}&\cdots x^-_{m_n},x^-_{s_1}x^-_{s_2}\cdots x^-_{s_{n-1}}\Omega_\psi(-m_1-p)x_{s_{n}}^- ) \label{ind}\\
&=\delta_{m_1+p,s_{n}}(x^-_{m_2}\cdots x^-_{m_n},x^-_{s_1}x^-_{s_2}\cdots x^-_{s_{n-1}})\in\mathbb Z[\![q]\!].\notag
\end{align}

Suppose \eqnref{omegaind} is true for $ i+1\leq n-1$. 
Then
\begin{align*}
(x^-_{m_2}&\cdots x^-_{m_n}x_{s_1}^-x^-_{s_2}\cdots  x^-_{s_i}\Omega_\psi(-m_1-t)x^-_{s_{i+1}}\cdots x_{s_n}^-) \\
&=\delta_{m_1+t,s_{i+1}}(x^-_{m_2}\cdots x^-_{m_n},x^-_{s_1}x^-_{s_2}\cdots x^-_{s_{i}} x^-_{s_{i+2}}\cdots s_{k_n}^-) \\
&\quad + \sum_{r\geq 0}g_{q^{-1}}(r)(x^-_{m_2}\cdots x^-_{m_n},x_{s_1}^-x^-_{s_2 }\cdots x^-_{s_{i} } x^-_{s_{i+1}+r}\Omega_\psi(-m_1-t-r)x^-_{s_{i+2}}\cdots x_{s_n}^-)  \\
&\equiv\delta_{m_1+t,s_{i+1}}(x^-_{m_2}\cdots x^-_{m_n},x^-_{s_1}x^-_{s_2}\cdots x^-_{s_{i}} x^-_{s_{i+2}}\cdots x_{s_n}^-) \\
&\quad +  g_{q^{-1}}(1)\delta_{m_1+t+1,s_{i+2}}\left(x^-_{m_2}\cdots x^-_{m_n},x_{s_1}^-x^-_{s_2}\cdots x^-_{s_{i}} x^-_{s_{i+1}+1}x^-_{s_{i+3}}\cdots x_{s_n}^-\right)  \mod  \mathbb Z[\![q]\!]  \\
&\equiv 0 \mod  \mathbb Z[\![q]\!] .
\end{align*}
Hence \eqnref{omegaind} is proved.

Now we want to prove a refined special case of \eqnref{omegaind}: For any $1\leq i\leq n-1$ and $t\in\mathbb Z_{\geq 0}$ one has
\begin{align}
(x^-_{m_2}&\cdots x^-_{m_n}x_{k_1+1}^-x^-_{k_2+1}\cdots  x^-_{k_{i-1}+1}\Omega_\psi(-m_1-t)x^-_{k_{i}}\cdots x_{k_n}^-)  \label{ind3}\\
&\equiv\delta_{m_1+t,k_{i}}(x^-_{m_2}\cdots x^-_{m_n},x^-_{k_1+1}x^-_{k_2+1}\cdots x^-_{k_{i-1}+1} x^-_{k_{i+1}}\cdots k_{k_n}^-) 
 \mod  q^2\mathbb Z[\![q]\!] .\notag
\end{align}
Here we assume $k_1\geq k_2\geq \cdots \geq k_n$.
For $i=n-1$ this is just \eqnref{ind}.  
Now for any $t\geq 0$ we assume that \eqnref{ind3} is true for $i+1\leq n-1$.  Then by \eqnref{omegaind} and induction we have 
\begin{align*}
(x^-_{m_2}&\cdots x^-_{m_n},x^-_{k_1+1}x^-_{k_2+1}\cdots x^-_{k_{i-1}+1}\Omega_\psi(-m_1-t)x_{k_{i}}^-\cdots x_{k_n}^-) \\
&=\delta_{m_1+t,k_{i}}(x^-_{m_2}\cdots x^-_{m_n},x^-_{k_1+1}x^-_{k_2+1}\cdots x^-_{k_{i-1}+1} x^-_{k_{i+1}}\cdots x_{k_n}^-) \\
&\quad + \sum_{r\geq 0}g_{q^{-1}}(r)(x^-_{m_2}\cdots x^-_{m_n},x_{k_1+1}^-x^-_{k_2+1}\cdots x^-_{k_{i-1}+1} x^-_{k_{i}+r}\Omega_\psi(-m_1-t-r)x^-_{k_{i+1}}\cdots x_{k_n}^-)  \\
&=\delta_{m_1+t,k_{i}}(x^-_{m_2}\cdots x^-_{m_n},x^-_{k_1+1}x^-_{k_2+1}\cdots x^-_{k_{i-1}+1} x^-_{k_{i+1}}\cdots x_{k_n}^-) \\
&\quad + g_{q^{-1}}(1)(x^-_{m_2}\cdots x^-_{m_n},x_{k_1+1}^-x^-_{k_2+1}\cdots x^-_{k_{i-1}+1} x^-_{k_{i}+1}\Omega_\psi(-m_1-t-1)x^-_{k_{i+1}}\cdots x_{k_n}^-)   \mod q^2\mathbb Z[\![q]\!] \\
&\equiv\delta_{m_1+t,k_{i}}(x^-_{m_2}\cdots x^-_{m_n},x^-_{k_1+1}x^-_{k_2+1}\cdots x^-_{k_{i-1}+1} x^-_{k_{i+1}}\cdots x_{k_n}^-) \\
&\quad +  g_{q^{-1}}(1)\delta_{m_1+t+1,k_{i+1}}\left(x^-_{m_2}\cdots x^-_{m_n},x_{k_1+1}^-x^-_{k_2+1}\cdots x^-_{k_{i-1}+1} x^-_{k_{i}+1}x^-_{k_{i+2}}\cdots x_{k_n}^-\right)  \mod q^2\mathbb Z[\![q]\!]  \\
&\equiv\delta_{m_1+t,k_{i}}(x^-_{m_2}\cdots x^-_{m_n},x^-_{k_1+1}x^-_{k_2+1}\cdots x^-_{k_{i-1}+1} x^-_{k_{i+1}}\cdots x_{k_n}^-) \\
&\quad +  g_{q^{-1}}(1)\delta_{m_1+t+1,k_{i+1}}\delta_{m_2,k_1+1}\cdots \delta_{m_{i+1},k_i+1}\delta_{m_{i+2},k_{i+2}}\cdots \delta_{m_n,k_n}  \mod q^2\mathbb Z[\![q]\!].
\end{align*}
where we used the fact that all monomials appearing are of PBW type with weakly decreasing indices. But the second summand in the last congruence above is nonzero only if 
$$
m_1 \geq m_{i+1}=k_i+1>k_i\geq k_{i+1}=m_1+t+1
$$
which is impossible for $t\geq 0$.  Hence the second summand is zero modulo $ \mod q^2\mathbb Z[\![q]\!]$.  This completes the proof of \eqnref{ind3}.

Now we show the induction step to complete the proof of the proposition:
\begin{align*} 
 (x_{m_1}&\cdots x_{m_n},x_{k_1}\cdots x_{k_n})= (x_{m_2}\cdots x_{m_n},\Omega_\psi(-m_1)x_{k_1}\cdots x_{k_n}) \\
 &=\delta_{m_1,k_1} (x_{m_2}\cdots x_{m_n}, x_{k_2}\cdots x_{k_n})   +\sum_{r\geq 0}g_{q^{-1}}(r) (x_{m_2}\cdots x_{m_n},x_{k_1+r}\Omega_\psi(-m_1-r)x_{k_2}\cdots x_{k_n}) \\
 &\equiv\delta_{\mathbf m,\mathbf k}     + g_{q^{-1}}(1) (x_{m_2}\cdots x_{m_n},x_{k_1+1}\Omega_\psi(-m_1-1)x_{k_2}\cdots x_{k_n})\mod q^2\mathbb Z[\![q]\!] \\
 &\equiv \delta_{\mathbf m,\mathbf k}     +g_{q^{-1}}(1) \delta_{m_1+1,k_2}(x_{m_2}\cdots x_{m_n},x_{k_1+1} x_{k_3}\cdots x_{k_n}) \mod q^2\mathbb Z[\![q]\!] \\
 &\equiv \delta_{\mathbf m,\mathbf k}     +g_{q^{-1}}(1) \delta_{m_1+1,k_2}\delta_{m_2,k_1+1}\delta_{m_3,k_3}\cdots \delta_{m_n,k_n} \mod q^2\mathbb Z[\![q]\!]\\
  \end{align*}
  where we used \eqnref{omegaind} in the third line and \eqnref{ind3} in the fourth line. The second summand in the last congruence is nonzero  if and only if $m_1+1=k_2,m_2=k_1+1,m_3=k_3,\dots m_n=k_n$. But this means that
$$
m_1\geq m_2=k_1+1>k_1\geq k_2=m_1+1
$$
which is a contradiction. This completes the proof of the proposition.

\end{proof}

\begin{cor}  The form $(\enspace,\enspace)$ is non-degenerate. 
\end{cor}
\begin{proof}  Suppose $u\in \mathcal N_q^-$, with $(u,v)=0$ for all $v\in \mathcal N_q^-$ and say $u=\sum_{\mathbf m}a_{\mathbf m}x_{m_1}\cdots x_{m_n}$, then in particular this holds for any $v=x_{k_1}\cdots x_{k_n}$.  Hence
$$
0=(u,x_{k_1}\cdots x_{k_n})=\sum_{\mathbf m}a_{\mathbf m}(x_{m_1}\cdots x_{m_n},x_{k_1}\cdots x_{k_n})=a_{\mathbf k}.
$$
Thus $a_{\mathbf k}=0$ for all $\mathbf k$.
\end{proof}

\section{ Imaginary Verma Modules for $A_1^{(1)}$}

We begin by recalling some basic facts and constructions for the affine
Kac-Moody algebra $ A_1^{(1)}$
and its imaginary Verma modules.
See \cite{K} for Kac-Moody algebra terminology and standard notations.

\subsection{}
The algebra $A_1^{(1)}$ is the affine Kac-Moody algebra over field $\mathbb F$  with
generalized Cartan matrix $A=(a_{ij})_{0\le i,j \le 1} = \begin{pmatrix}2&-2 \\ -2&\
2 \\ \end{pmatrix}$.
The algebra $A_1^{(1)}$ has a Chevalley-Serre presentation with generators
$e_0, e_1, f_0, f_1, h_0, h_1, d$ and relations
\begin{align*}
&[h_i,h_j] =0, \ \ [h_i,d]=0, \\
&[e_i,f_j] = \delta_{ij}h_i, \\
&[h_i, e_j] = a_{ij}e_j, \ \ [h_i, f_j] = -a_{ij}f_j, \\
&[d, e_j]= \delta_{0,j} e_j, \ \ [d, f_j]=-\delta_{0,j} f_j, \\
&(\text{ad}\, e_i)^3e_j = (\text{ad}\, f_i)^3f_j = 0, \quad i \neq j.
\end{align*}
Alternatively, we may realize $A_1^{(1)}$ through the loop algebra
construction
$$
A_1^{(1)}
\cong {\mathfrak sl}_2 \otimes \mathbb F[t,t^{-1}] \oplus \mathbb F c \oplus
\mathbb F d
$$
with Lie bracket relations
\begin{align*}[x\otimes t^n,y\otimes t^m]&= [x,y] \otimes t^{n+m}
+ n \delta_{n+m,0}(x,y)c, \\
[x \otimes t^n,c] = 0 = [d,c] , \quad & \quad [d,x \otimes t^n] = nx \otimes t^n,
\end{align*}for $x,y \in {\mathfrak sl_2}$, $n,m \in \mathbb Z$, where $(\ , \ )$ denotes the
Killing form on ${\mathfrak sl_2}$.
For $x \in {\mathfrak sl_2}$ and $n \in \mathbb Z$, we write $x(n)$ for
$x \otimes t^n$.

Let $\Delta$ denote the root system of $A_1^{(1)}$, and let
$\{ \alpha_0, \alpha_1\}$ be a basis for $\Delta$.  Let $\delta = \alpha_0 + \alpha_1$,
the minimal imaginary root.  Then
$$
\Delta = \{ \pm \alpha_1 + n\delta\ |\ n \in \mathbb Z\} \cup \{ k\delta\ |\ k \in \mathbb Z
\setminus \{ 0 \} \}.
$$

\subsection{}
The universal enveloping
algebra $
U(A_1^{(1)})$ of $A_1^{(1)}$
is the associative algebra
over $\mathbb F$ with 1
generated by the elements $h_0, h_1, d, e_0, e_1, f_0, f_1$
with defining relations
\begin{align*}&[h_0,h_1]=[h_0,d] = [h_1,d]=0, \\
&h_ie_j-e_jh_i = a_{ij}e_j, \quad h_if_j-f_jh_i=-a_{ij}f_j, \\
&d e_j-e_j d =\delta_{0,j}e_j, \quad d f_j-f_j d = -\delta_{0,j}f_j, \\
&e_if_j-f_je_i = \delta_{ij}h_i, \\
&e_je_i^3-3e_ie_je_i^2+3e_i^2e_je_i-e_i^3e_j = 0 \text{ for } i \neq j, \\
&f_jf_i^3-3f_if_jf_i^2+3f_i^2f_jf_i-f_i^3f_j = 0 \text{ for } i \neq j.
\end{align*}
 Corresponding to the loop algebra formulation of $A_1^{(1)}$ is an
alternative description of
$U(A_1^{(1)})$ as the associative algebra over $\mathbb F$ with 1 generated
by the elements $e(k), f(k)$ $(k \in \mathbb Z)$, $h(l)$ $(l \in
\mathbb Z\setminus \{ 0\})$,
$c, d, h$, with relations
\begin{align*}& [c,u]=0 \ \ \text {for all} \ u\in U(A_1^{(1)}), \\
& [h(k), h(l)]=2k \delta_{k+l,0} c, \\
& [h,d]=0, \ \ [h, h(k)]=0, \\
& [d,h(l)]=l h(l), \ \ [d,e(k)]=ke(k), \ \ [d,f(k)]=kf(k),\\
& [h,e(k)]=2e(k), \ \ [h,f(k)]=-2f(k), \\
& [h(k), e(l)]=2e(k+l), \ \ [h(k), f(l)]=-2f(k+l), \\
& [e(k), f(l)]=h(k+l)+k \delta_{k+l,0} c.
\end{align*}

\subsection{}  A
subset $S$ of the root system $\Delta$ is called {\it closed}
if $\alpha, \beta \in S$ and
$\alpha+\beta \in \Delta$
implies $\alpha+\beta \in S$.  The subset $S$ is called a {\it closed
partition } of the roots if $S$ is closed,
$S \cap(-S) = \emptyset$, and $S\cup -S = \Delta$ \cite{JK},\cite{MR89m:17032},\cite{MR1078876},\cite{MR1175820}.
The set
$$
S= \{ \alpha_1+k\delta \ |\ k\in \mathbb Z \} \cup \{l\delta\ |\ l \in \mathbb Z_{>0} \}
$$
is a closed partition of $\Delta$ and is $W\times
\{\pm{1}\}$-inequivalent to the standard
partition of the root system into positive and negative roots \cite{MR95a:17030}.

For $\hat{\mathfrak g}= A_1^{(1)}$, let ${\mathfrak g}_{\pm}^{(S)}=\sum_{\alpha \in S}
\hat{\mathfrak g}_{\pm \alpha}$.  In the loop algebra formulation of $\hat{\mathfrak g}$, we have
that ${\mathfrak g}_+^{(S)}$ is the
subalgebra generated by $e(k)$ $(k \in \mathbb Z)$ and $h(l)$ $(l\in \mathbb Z_{>0})$
and ${\mathfrak g}_-^{(S)}$ is the subalgebra generated by $f(k)$ $(k \in \mathbb Z)$ and
$h(-l)$ $(l\in \mathbb Z_{>0})$.  Since $S$ is a partition of the root system,
the algebra has a direct
 sum decomposition
$$
\hat{\mathfrak g}={\mathfrak g}_{-}^{(S)} \oplus {\mathfrak h} \oplus {\mathfrak g}_{+}^{(S)}.
$$
Let $U({\mathfrak g}_{\pm}^{(S)})$ be the universal enveloping algebra of
${\mathfrak g}_{\pm}^{(S)}$. Then, by the PBW theorem, we have
$$
U(\hat{\mathfrak g}) \cong U({\mathfrak g}_{-}^{(S)}) \otimes U({\mathfrak h})\otimes U({\mathfrak g}_{+}^{(S)}),
$$
where $U({\mathfrak g}_{+}^{(S)})$ is generated by $ e(k)$ $(k\in \mathbb Z)$, $h(l)$
$(l\in \mathbb Z_{>0})$,
$U({\mathfrak g}_{-}^{(S)})$ is generated by $f(k)$ $(k\in \mathbb Z)$, $h(-l)$ $(l\in
\mathbb Z_{>0})$ and $U({\mathfrak h})$,
the universal enveloping algebra of ${\mathfrak h}$, is generated by
$h,c$ and $d$.

Let $\lambda\in P$, the weight lattice of $\hat{\mathfrak g}=A_1^{(1)}$.
A $U(\hat{\mathfrak g})$-module $V$ is called a {\it weight} module if
$V=\oplus_{\mu \in P} V_{\mu}$, where
$$
V_{\mu}=\{ v \in V\ |\ h\cdot v=\mu(h)v, c\cdot v=\mu(c)v,
d\cdot v = \mu(d)v \}.
$$
Any submodule of a weight module is a weight module.
A $U(\hat{\mathfrak g})$-module $V$ is
called an {\it $S$-highest weight module}
with highest weight $\lambda$ if there is a
non-zero $v_{\lambda} \in V$ such that
(i) $u^+ \cdot v_{\lambda} = 0$ for all $u^+\in U({\mathfrak g}_{+}^{(S)})
\setminus \mathbb F^*$, (ii) $h\cdot v_{\lambda}=\lambda(h)v_{\lambda}$, $c\cdot v_{\lambda}
= \lambda(c)v_{\lambda}$, $d\cdot v_{\lambda}  = \lambda(d)v_{\lambda}$,
(iii) $V=U(\hat{\mathfrak g})\cdot v_{\lambda} = U({\mathfrak g}_{-}^{(S)}) \cdot v_{\lambda}$.
An $S$-highest weight module is a weight module.

For $\lambda \in P$, let $I_S(\lambda)$ denote the ideal of $U(A_1^{(1)})$
generated by
$e(k)$ $(k\in \mathbb Z)$, $h(l)$ $(l>0)$, $h-\lambda(h) 1$,
$c-\lambda(c) 1$, $d-\lambda(d) 1$.
Then we define $M(\lambda) = U(A_1^{(1)})/I_S(\lambda)$ to be the {\it
imaginary Verma module} of
$A_1^{(1)}$ with highest weight $\lambda$.
Imaginary Verma modules have many structural features similar to those of
standard
Verma modules,
with the exception of the infinite-dimensional weight spaces.
Their properties were investigated in \cite{MR95a:17030}, from which we recall
 the following
proposition \cite[ Proposition 1, Theorem 1]{MR95a:17030}.

\begin{prop} (i) $M(\lambda)$ is a $U({\mathfrak g}_-^{(S)})$-free
module of rank 1 generated by
the $S$-highest weight vector $1\otimes 1$ of weight $\lambda$.\newline
(ii) $\dim M(\lambda)_{\lambda} =1$; $0<\dim M(\lambda)_{\lambda - k \delta} < \infty$
for any integer $k>0$; if $\mu \neq \lambda- k\delta$ for any integer
 $k \ge 0$ and
$M(\lambda)_{\mu} \neq 0$, then $\dim M(\lambda)_{\mu} = \infty$.\newline
(iii) Let $V$ be a $U(A_1^{(1)})$-module generated by some $S$-highest
weight vector $v$
 of weight $\lambda$.  Then there exists a unique surjective homomorphism
$\varphi:M(\lambda) \to V$ such that $\varphi(1 \otimes 1) =v$. \newline
(iv) $M(\lambda)$ has a unique maximal submodule. \newline
(v) Let $\lambda, \mu \in P$.  Any non-zero element of
$\text{Hom}\,_{U(A_1^{(1)})}(M(\lambda),M(\mu))$ is
injective.\newline
(vi) $M(\lambda)$ is irreducible if and only if $\lambda(c)\neq 0$.
\hskip 1cm $\square$
\end{prop}

Suppose now that $\lambda(c)=0$.  Consider an ideal
$J(\lambda)$ of  $U(A_1^{(1)})$ generated by $I_S(\lambda)$ and  $h(l)$ for
all $l$. Set 
$$\tilde{M}(\lambda)=U(A_1^{(1)})/J(\lambda).$$
Then $\tilde{M}(\lambda)$ is a homomorphic image of
 $M(\lambda)$ which we call the \emph{reduced imaginary Verma
 module}. The module $\tilde{M}(\lambda)$ has a $\Lambda$-gradation:
 $$\tilde{M}(\lambda)=\sum_{\xi\in\Lambda}\tilde{M}(\lambda)_{\xi}.$$
 

\section{The category $\mathcal O_{\text{red,im}}$}  
Let $G$ be the Heisenberg subalgebra, $G=\sum_{k\in \mathbb Z\setminus \{0\}}\hat{\mathfrak g}_{k\delta}\oplus \mathbb F c$. We say that a nonzero $\hat{\mathfrak g}$-module $V$ is $G$-compatible if 
\begin{enumerate}[i).]
\item  \label{1}  $V$ has a decomposition $V=TF(V)\oplus T(V)$ into a sum of nonzero $G$-submodules such that
\item  \label{2} 
 $G$ is bijective on $TF(V)$ (that is, any nonzero element $g\in G$ is a bijection on $TF(V)$) and $TF(V)$ has no nonzero $\hat{\mathfrak g}$-submodule, 
 \item \label{3}     $G\cdot T(V)=0$. 
\end{enumerate}
Consider the set 
$$
\mathfrak h^*_{red}:=\{\lambda\in\mathfrak h^*\,|\, \lambda(c)=0,\lambda (h)\notin \mathbb Z_{\geq 0}\}.
$$
As usual let
$$
e=\begin{pmatrix} 0 & 1\\ 0 & 0 \end{pmatrix},\quad f=\begin{pmatrix} 0 & 0\\ 1 & 0 \end{pmatrix},\quad h=\begin{pmatrix} 1 & 0\\ 0 & -1 \end{pmatrix}.
$$
The category $\mathcal O_{\text{red,im}}$ has as objects $\hat{\mathfrak g}$-modules $M$ such that 
\begin{enumerate}
\item  $$
M=\bigoplus_{\nu\in\mathfrak h^*_{red}}M_\nu,\quad\text{ where }\quad M_\mu=\{m\in M\,|\, hm=\nu(h)m\}.
$$ 
Note $\dim M_\nu$ may be infinite dimensional.
\item $e_n =e\otimes t^n$ acts locally nilpotently for any $n\in \mathbb Z$.
\item $M$ is $G$-compatible.

\end{enumerate}
The morphisms in the category are  $\hat{\mathfrak g}$-module homomorphisms.  For example direct sums of reduced imaginary Verma modules $ \bar{M}(\lambda)$ are in the category $\mathcal O_{\text{red,im}}$.
In this case $TF(\bar{M}(\lambda))=\oplus_{k\in \mathbb Z, n\in \mathbb Z_{>0}}\bar{M}(\lambda)_{\lambda-n\alpha+k\delta}$ and 
 $T(\bar{M}(\lambda))=\bar{M}(\lambda)_{\lambda}\simeq \mathbb F$.

A loop module for $\mathfrak g$ is any representation of the form $\hat M:=M\otimes \mathbb F[t,t^{-1}]$ where $M$ is a highest weight module for $\mathfrak{sl}(2,\mathbb F)$ and 
$$
(x\otimes t^k)(m\otimes t^l):=x\cdot m\otimes t^{k+l},\quad c(m\otimes t^l)=0.
$$
Here $x\cdot m$ is the action of $x\in \mathfrak{sl}(2,\mathbb F)$ on $m\in M$.
\begin{prop}  
\begin{enumerate} 
\item   The loop modules $\hat M$ with $M$ in the category $\mathcal O$ for $\mathfrak{sl}(2,\mathbb F)$ are not in $\mathcal O_{\text{red,im}}$.
\item  For  $\lambda,\mu\in\mathfrak h^*_{red}$ one has  $\text{Ext}^1_{\hat{\mathfrak g}}(\bar M(\lambda),\bar M(\mu))=0$.
\end{enumerate}
\end{prop}
\begin{proof} Suppose $\hat M$ is a loop module with $M\in \mathcal O$ for $\mathfrak{sl}(2,\mathbb F)$.   Then $\hat M$ satisfies condition (1) and (2) from above.   Assume $\hat M=TF(\hat M)\oplus T(\hat M)$ satisfies \eqnref{1}-\eqnref{3} above.   Now take any $\sum_{i=-k}^km_i\otimes t^i\in T(\hat M)$ with $m_i\in M_\mu$ for some weight $\mu$.  Then by \eqnref{3} we have 
$$
0=h\otimes t^r\cdot \left(\sum_{i=-k}^km_i\otimes t^i\right)=\lambda(h)\left(\sum_{i=-k}^km_i\otimes t^{i+k}\right)
$$
so that $\lambda(h)=0$ which contradicts $\lambda\in \mathfrak h^*_{red}$.  Then $T(\hat M)=0$ and $\hat M=TF(\hat M)$ which is a $\hat{\mathfrak g}$-module contradicting \eqnref{1} and \eqnref{2} and thus (3).



For (2) we need to show that there are no nontrivial extensions between reduced imaginary Verma modules 
 $ \bar{M}(\lambda)$ and $ \bar{M}(\mu)$.  
If $\mu=\lambda+k\delta$ for some integer $k$ then any extension of $\bar M(\lambda)$ by itself has a two dimensional highest weight space of weight $\lambda$.  Any highest weight vector in this space generates an irreducible submodule and thus the extension splits as a direct sum of two submodules each isomorphic to $\bar M(\lambda)$. 

Indeed suppose now $\mu=\lambda+k\delta-s\alpha$ 
for some integers $k$ and $s>0$.
Consider a short exact sequence 
\begin{equation}\label{ses} 
\begin{CD}
0 @>>> \bar M(\lambda)@>\iota >>  M@>\pi>>  \bar M(\mu)@>>> 0,
\end{CD}
\end{equation} 
where we view $\iota$ as just the inclusion map. For any preimage weight vector $\bar v_\mu$ of a highest weight (w.r.t. $\mathfrak{sl}(2,\mathbb F)$) vector $v_\mu$ in $\bar M(\mu)$ one has  $G\bar v_\mu\in\bar M(\lambda)$.
 On the other hand $G\bar v_{\mu}=0$. Suppose $0\neq  v= h_m\bar v_{\mu}$.  Then $h_m\bar v_{\mu}=h_mv'$ for some $v'\in \bar M(\lambda)$ (one cannot have $h_m\bar v_\mu=\alpha v_\lambda$, $\alpha\in\mathbb F$ as otherwise $\mu+m\delta=\lambda$ and $s=0$).   Then $h_n(v'-\bar v_\mu)=0$ and so $v'-\bar v_\mu\in T(M)=\mathbb Fv_\lambda$ which is a contradiction to the fact  $\bar v_\mu\not\in \bar M(\lambda)$.

Recall that $e_0$ acts locally nilpotently on $\bar v_{\mu}$. 
Moreover, $e_0^t \bar v_{\mu}\neq 0$ if $t<s$, otherwise $e_0^{t-1} \bar v_{\mu}$ would generate a submodule in $ \bar{M}(\lambda)$ which is a contradiction. 
So, $e_0^s \bar v_{\mu}=0$ if $k=0$ and $e_0^{s-1} \bar v_{\mu}=0$ if $k\neq 0$. Without loss of generality we assume the latter. Suppose $s>1$.
Consider an $\mathfrak{sl}(2)$-subalgebra $\mathfrak{a}$ generated by $e_0$ and $f_0$ and an  $\mathfrak{a}$-module generated by $\bar v_{\mu}$. This module is a non trivial extension of two Verma modules over $\mathfrak{a}$ with highest weights $\lambda +k\delta - \alpha$ and $\lambda +k\delta -s\alpha$. But this is impossible (e.g. these modules have different central characters). Suppose now $s=1$. Then apply 
the same argument to  an $\mathfrak{sl}(2)$-subalgebra  generated by $e_{k}$ and $f_{k}$. Assuming $e_{k}\bar v_{\mu}\neq 0$ we obtain a contradiction as above. 
Therefore, $M= \bar{M}(\lambda)\oplus  \bar{M}(\mu)$ completing the proof.  

\end{proof}
 
 \begin{prop}\label{prop-irred}
If $M\in \mathcal O_{\text{red,im}}$ is a simple object, then $M\simeq \bar{M}(\lambda)$ for some $\lambda\in \mathfrak h^*_{red}$.

\end{prop}

\begin{proof}
Consider any simple $M\in \mathcal O_{\text{red,im}}$. Let $v\in T(M)$ be a nonzero element of some weight 
$\lambda\in \mathfrak h^*_{red}$.  Then $Gv=0$ and $e_0^Nv=0$ for some positive integer $N$. Choose $N$ to be the least possible with 
such property. If $N=1$ then $e_nv=0$ for all integers $n$ and hence $M$ is a quotient of the reduced imaginary Verma module $ \bar{M}(\lambda)$ with highest weight $\lambda$. Since  $\lambda\in \mathfrak h^*_{red}$ then $ \bar{M}(\lambda)$ is simple and thus $M\simeq  \bar{M}(\lambda)$. Assume now that $N>1$ and set $w=e_0^{N-1}v$. Then $e_0w=0$. We have $0=h_{k\delta}e_0^Nv=2N\color{black}e_k e_0^{N-1}v=2Ne_k w$ for all integers $k$. Therefore $M$ is a quotient of the loop module induced from $U(G)w$ (with $e_kU(G)w=0$ for all integers $k$).  If $w\in T(M)$ then we are done. Suppose $w\not\in T(M)$ so $0\neq w\in TF(M)$ may be assumed to be a weight vector of weight $\mu$.   Then $W = U(G)w$ is a $G$ -submodule of $TF(M)$.
Consider the induced module 
$I(W)=Ind_{G+N_++H}^{\hat\hat{\mathfrak g}} W$ where $N_+=\oplus_{n\in\mathbb Z}\mathbb Fe_n$ acts by zero on $W$, $H=\mathbb Fh+\mathbb Fd$ acts by $hw=\mu(h)w$ and $dw=\mu(d)w$. Since $U(G)w\subset TF(M)$   then it is easy to see that $TF(I(W))=I(W)$. Hence the same 
holds for any of its quotients by the Short Five Lemma, i.e. $TF(M)=M$ which is a contradiction. Therefore $w\in T(M)$ which completes the proof. 
\end{proof}

\begin{thm}\label{thm-decomp}
If $M\in \mathcal O_{\text{red,im}}$ is any object then $M=\oplus_{\lambda_i\in \mathfrak h^*_{red}}\bar{M}(\lambda_i)$, $i\in I$
for some weights $\lambda_i$'s.
\end{thm}

\begin{proof}
Consider the subspace $T(M)$. Since the weights of $M$ are in $\mathfrak h^*_{red}$, $T(M)$ is not a $\hat{\mathfrak g}$-submodule. 
Let $w\in T(M)$ be a nonzero element,  $W=U(G)w\subset T(M)$. Arguing as in the proof of Proposition \ref{prop-irred} we find a nonzero element $w'\in M$ such that 
  $e_k w'=0$ for all integers $k$. If $U(G)w'\neq \mathbb C w'$ then $w'\in TF(M)$ which is a contradiction. Hence $w'$ generates a submodule isomorphic to a reduced imaginary Verma module containing $W$. Thus each nonzero element of $T(M)$ generates $ \bar{M}(\lambda)$ for some $\lambda$. 
\end{proof}
\begin{cor}  The category $\mathcal O_{\text{red,im}}$ is closed under taking subquotients and direct sums so it is a Serre category. 
\end{cor}

\section{Quantized Imaginary Verma modules}
Let $\Lambda$ denotes the weight lattice of $\hat{\mathfrak g} =A_1^{(1)}$,
$\lambda\in \Lambda$. Denote by $I^q(\lambda)$ the ideal of
$U_q=U_q(\hat{\mathfrak g}$) generated by $x^+(k)$, $k\in
\mathbb Z$, $a(l), l>0$, $K^{\pm 1}-q^{\lambda(h)}1$, $\gamma^{\pm
\frac{1}{2}}-q^{\pm \frac{1}{2}\lambda(c)}1$ and $D^{\pm 1}-q^{\pm
\lambda(d)}1$. The imaginary Verma module with highest weight
$\lambda$ is defined to be (\cite{MR97k:17014}) $$M_q(\lambda)=U/I^q(\lambda).$$

\begin{thm}[\cite{MR97k:17014}, Theorem 3.6]  The imaginary Verma module $M_q(\lambda)$ is simple
 if and only if $\lambda(c)\neq 0$.
\end{thm}

Suppose now that $\lambda(c)=0$. Then $\gamma^{\pm \frac{1}{2}}$ acts on $M_q(\lambda)$ by $1$.  Consider the ideal
$J^q(\lambda)$ of  $U_q$ generated by $I^q(\lambda)$ and  $a(l)$ for
all $l$. Denote
$$\tilde{M}_q(\lambda)=U_q/J^q(\lambda).$$
Then $\tilde{M}_q(\lambda)$ is a homomorphic image of
 $M_q(\lambda)$ which we call the \emph{reduced quantized imaginary Verma
 module}. The module $\tilde{M}_q(\lambda)$ has a $\Lambda$-gradation:
 $$\tilde{M}_q(\lambda)=\sum_{\xi\in\Lambda}\tilde{M}_q(\lambda)_{\xi}.$$
 
 \begin{thm}[\cite{CFM10}]
Let $\lambda\in \Lambda$ be such that $\lambda(c)=0$. Then module $\tilde{M}_q(\lambda)$
is simple if and only if $\lambda(h)\neq 0$.
\end{thm}

\section{The category $\mathcal O^q_{\text{red,im}}$}  Consider the set 
$$
\mathfrak h^*_{red}:=\{\lambda\in\mathfrak h^*\,|\, \lambda(c)=0,\lambda (h)\neq 0\}.
$$
The category $\mathcal O^q_{\text{red,im}}$ has as objects $U_q(\hat{\mathfrak g})$-modules $M$ such that 
there exists $\lambda_i\in \mathfrak h^*_{red}$, $i\in I$,  with 
$$
M\cong \bigoplus_{i\in I}\tilde M_q(\lambda_i).
$$
The morphisms in the category are just $U_q(\hat{\mathfrak g})$-module homomorphisms. 
Since $\bar M_q(\lambda)$ is a quantization of $\bar M(\lambda)$ in the sense of Lusztig, modules in  $\mathcal O^q_{\text{red,im}}$ are quantizations of modules in $\mathcal O_{\text{red,im}}$.
So equivalently the category $\mathcal O^q_{\text{red,im}}$ can be defined as follows:

Let $G_q$ be the quantized Heisenberg subalgebra generated by $h_k, k\in \mathbb Z\setminus \{0\}$ and $\gamma$. We say that a nonzero $U_q(\hat{\mathfrak g})$-module $V$ is $G_q$-compatible if 
\begin{enumerate}[i).]
\item  \label{1}  $V$ has a decomposition $V=TF(V)\oplus T(V)$ into a sum of nonzero $G_q$-submodules such that
\item  \label{2} 
 $G_q$ is bijective on $TF(V)$ (that any nonzero element $g\in G_q$ is a bijection on $TF(V)$) and $TF(V)$ has no nonzero $U_q(\mathfrak g)$-submodule, 
 \item \label{3}     $G_q\cdot T(V)=0$. 
\end{enumerate}
The category $\mathcal O^q_{\text{red,im}}$ has as objects $U_q(\hat{\mathfrak  g})$-modules $M$ such that 
\begin{enumerate}
\item $$
M=\bigoplus_{\nu\in\mathfrak h^*_{red}}M_\nu,\quad\text{ where }\quad M_\nu=\{m\in M\,|\, Km=K^{\nu(h)}m,\enspace Dm=q^{\nu(d)}m\},
$$
\item $x^+_n$, $n\in\mathbb Z$ act locally nilpotently,
\item $M$ is $G_q$-compatible.

\end{enumerate}

If $M\in\mathcal O^q_{\text{red,im}}$, we can write $M=\oplus_i\tilde M_q(\lambda_i)$ with $\tilde M_q(\lambda_i)=\oplus \mathbb F(q^{1/2})x^-_{n_1}\cdots x^-_{n_k}v_{\lambda_i}$.  We define $\tilde\Omega_\psi(m)$ and $\tilde x_m^-$ on each $\tilde M_q(\lambda_i)$ as in \eqnref{definingomegapsi}: 
\begin{align}
\tilde\Omega_\psi(m)(x^-_{n_1}\cdots x^-_{n_k}v_{\lambda_i})&:=\Omega_\psi(m)(x^-_{n_1}\cdots x^-_{n_k})v_{\lambda_i} \\
\tilde x_m^-(x^-_{n_1} \cdots x^-_{n_k}v_{\lambda_i})&:=  x_m^-x^-_{n_1}\cdots x^-_{n_k}v_{\lambda_i}.
\end{align}
Hence the following result follows.
  \begin{thm} The operators $\tilde\Omega_\psi(m)$ and $\tilde x_m^-$ are well defined on objects in the category $\mathcal O^q_{\text{red,im}}$. Moreover on each summand $\tilde M_q(\lambda_i)\cong \mathcal N_q^-$ they agree with the $\Omega_\psi(m)$ respectively left multiplication by $x_m^-$ defined as in \eqnref{definingomegapsi}. 
\end{thm}
\section{Imaginary $\mathbb A$-lattices and imaginary crystal basis}  Let $\mathbb A_0$ (resp. $\mathbb A_\infty$) to be the ring of rational functions in $q^{1/2}$ with coefficients in a field $\mathbb F$ of characteristic zero, regular at $0$ (resp. at $\infty$).  
 Let $\mathbb A= \mathbb F[q^{1/2},q^{-1/2}, \frac{1}{[n]_{q}}, n>1]$, and $P=\{-k\alpha+m\delta\,|\,k>0,m\in\mathbb Z\}\cup\{0\}$.    Let $M$ be a $U_q(\hat{\mathfrak g})$-module in the category. We call a free $\mathbb A_0$-submodule $\mathcal L$ of $M$ an {\it imaginary crystal $\mathbb A_0$-lattice} of $M$ if the following hold
\begin{enumerate}[(i).]
\item $\mathbb F(q^{1/2})\otimes_{\mathbb A_0}\mathcal L\cong M$\label{lattice1},
\item $\mathcal L=\oplus_{\lambda \in P}\mathcal L_\lambda$ and $\mathcal L_\lambda =\mathcal L\cap \mathcal M_\lambda$,
\item $\tilde\Omega_\psi(m)\mathcal L\subseteq \mathcal L$ and $\tilde x^-_m\mathcal L\subseteq \mathcal L$ for all $m\in\mathbb Z$.
\end{enumerate}

We now show that the above definition is not vacuous.  Let $\lambda\in \mathfrak h^*$ and define
\begin{align*}
\mathcal L(\lambda):&=\sum_{k\geq 0\atop  i_1\geq \dots\geq i_k, i_j\in\mathbb Z}\mathbb Ax^-_{i_1}\cdots x^-_{i_k}v_\lambda\subset \mathcal N_q^-v_\lambda=\tilde M_q(\lambda)
\end{align*}
and also operators $\tilde \Omega_\psi(m):\tilde M_q(\lambda)\to \tilde M_q(\lambda)$ and $\tilde x_m^-:\tilde M_q(\lambda)\to\tilde M_q(\lambda)$ where $\tilde x^-_m$ is the left multiplication  operator by $x_m^-$ and $\tilde \Omega_\psi(m)(x^-_{i_1}\cdots x^-_{i_k}v_\lambda):=\Omega_\psi(m)(x^-_{i_1}\cdots x^-_{i_k})v_\lambda$ for $i_1\geq \cdots \geq i_k$.

For $\mu =\lambda-k\alpha+m\delta$, 
\begin{align*}
\tilde M_q(\lambda)_\mu=\begin{cases}\bigoplus_{\sum_{j=1}^ki_j=m,i_1\geq \cdots \geq  i_k}\mathbb Q(q^{1/2})x^-_{i_1}\cdots x^-_{i_k}v_\lambda&\quad \text{ if }k>0,  \\
\mathbb Q(q^{1/2}) v_\lambda&\quad\text{if }k=0
\end{cases}
\end{align*}
Now observe \eqnref{lattice1} is satisfied for $\mathcal L=\mathcal L(\lambda)$ as well as 
\begin{enumerate}
\item for $\mathcal L(\lambda)_\mu:=\mathcal L(\lambda)\cap \tilde M_q(\lambda)_\mu$ one has $\mathcal L(\lambda)=\bigoplus_{\lambda \in P}\mathcal L(\lambda)_{\mu}$, 
and
\item \begin{align*}
\tilde x^-_m\mathcal L(\lambda)\subseteq \mathcal L(\lambda),  \qquad \text{and}\qquad
\tilde\Omega_\psi(m)\mathcal L(\lambda)\subseteq \mathcal L(\lambda)
\end{align*}
where first statement follows from \eqnref{Serre1} and \eqnref{Serre2} and the last statement follows from \eqnref{Serre1}, \eqnref{Serre2}, \eqnref{omegapsi6} and the fact that $g_q(r)\in\mathbb A$ for $r\in\mathbb Z$ by 
\eqnref{grcomp}.
Thus $\mathcal L(\lambda)$ is an imaginary crystal lattice. 
\begin{prop}  
$$
\mathcal L(\lambda)=\left\{u\in \tilde M_q(\lambda)\,|\, (u,\tilde M_q(\lambda))\subset \mathbb A_0\right\}
$$
If $\mathbb F=\mathbb Q$, then 
$$
\mathcal L(\lambda)=\left\{u\in \tilde M_q(\lambda)\,|\, (u,u)\in \mathbb A_0\right\}
$$
\end{prop}
\begin{proof}  Let $R$ denote the right hand side of the above equality. 
We have the inclusion $\mathcal L(\lambda)\subseteq R$ by \propref{symmetricform}.  For the other inclusion let $u\in R$ and by clearing denominators we can find a smallest $n\geq  0$ such that $q^{n/2}u\in \mathcal L(\lambda)$.  If $n>1$ then
$$
(q^{n/2}u,\tilde M_q(\lambda))\equiv 0\mod q^{n/2}\mathbb A_0.
$$
By \propref{symmetricform} $(\enspace,\enspace)$ is non-degenerate modulo $q^2\mathcal L(\lambda)$, we must have $q^{n/2}u\equiv 0\mod q^2\mathcal L(\lambda)$. Hence $q^{(n/2)-2}u\in \mathcal L(\lambda)$ which contradicts the minimality of $n$.  Thus $u\in \mathcal L(\lambda)$.  
\end{proof}
\end{enumerate}

For $\lambda \in\mathfrak h^*$ define 
$$
\mathcal B(\lambda):=\left\{\tilde x_{i_1}^-\cdots \tilde x_{i_k}^-v_\lambda +q\mathcal L(\lambda)\in \mathcal L(\lambda)/q\mathcal L(\lambda)\,|\, i_1\geq \cdots \geq  i_k\right\}.
$$

An {\it imaginary crystal basis} of a $U_q(\hat{\mathfrak g})$-module $M$ in the category $\mathcal O^q_{\text{red,im}}$ is a pair $(\mathcal L,\mathcal B)$  satisfying
\begin{enumerate}[(i).]
\item $\mathcal L$ is an imaginary crystal lattice of $M$,
\item $\mathcal B$ is an $\mathbb F$-basis of $\mathcal L/q\mathcal L\cong \mathbb F\otimes_{\mathcal A_0}\mathcal L$,
\item $\mathcal B=\cup_{\mu \in P}\mathcal B_\mu$, where $\mathcal B_\mu =\mathcal B\cap (\mathcal L_\mu/q\mathcal L_\mu)$, 
\item $\tilde x_m^-\mathcal B\subset\pm \mathcal B\cup\{0\}$ and $\tilde \Omega_\psi\mathcal B\subset  \pm\mathcal B\cup \{0\}$,
\item For $m\in\mathbb Z$, if $ \Omega_\psi(-m)b\neq 0$ and $\tilde x_m^-b\neq0$ for $b\in \mathcal B$, then $\tilde x_m^-\tilde \Omega_\psi(-m)b=\tilde \Omega_\psi(-m)\tilde x_m^-b$.  .
\end{enumerate}

\begin{thm}  For $\lambda\in \hat{\mathfrak h}^*_{red,im}$, the pair $(\mathcal L(\lambda),\mathcal B(\lambda))$ is an imaginary crystal basis of the reduced imaginary Verma module $\tilde M_q(\lambda)$.  
\end{thm}
\begin{proof} Conditions (i)-(ii) are clear. For (iii) consider $b=\tilde x_{i_1}^-\cdots \tilde x_{i_k}^-v_\lambda +q\mathcal L(\lambda)$ with $i_1\geq i_2\geq \cdots \geq i_k$.  If $m\geq i_1$, then
$$
\tilde x_m^-b=\tilde x_m^-\tilde x_{i_1}^-\cdots \tilde x_{i_k}^-v_\lambda +q\mathcal L(\lambda)\in \mathcal B.
$$
If $m= i_1-1$, then by \eqnref{Serre1} we have
$$
\tilde x_m^-b=q^2\tilde x_m^-\tilde x_{i_1}^-\cdots \tilde x_{i_k}^-v_\lambda +q\mathcal L(\lambda)=0\mod q\mathcal L(\lambda).
$$
If $m<i_1-1$, ($l=m$ and $k+1=i_1$ so $k=i_1-1$) then by \eqnref{Serre2} we have
\begin{align} 
\tilde x_m^-b&\equiv-\tilde x_{i_1-1}^-\tilde x_{m+1}^-\tilde x_{i_2}^-\cdots \tilde x_{i_k}^-v_\lambda +q\mathcal L(\lambda)   
\end{align}
and $i_1-1\geq m+1$.  By induction this is either $0 \mod q\mathcal L(\lambda)$ if $i_j=m+j$ for some $j$ or in $\pm \mathcal B$.
To sum it up we have
\begin{align}\label{xonbasis}
\tilde x_m^-b&\equiv \begin{cases}
\tilde x_m^-\tilde x_{i_1}^-\cdots \tilde x_{i_l}^-v_\lambda +q\mathcal L(\lambda)&\text{ if }m\geq i_1\\
0 & \text{ if }m+j= i_j\quad \text{ for some }1\leq j\leq l ,\\
(-1)^{j-1}\tilde x_{i_1-1}^-\tilde x_{i_2-1}^- \cdots \tilde x_{m+j-1}^-\tilde x_{i_j}^-\cdots \tilde x_{i_l}^-v_\lambda +q\mathcal L(\lambda) & \text{ if }m+j>   i_j\text{ but }m+j-1< i_{j-1},\\  &\text{ for some }1\leq j\leq l.
\end{cases}  \\
&\equiv \begin{cases}
0 & \text{ if }m+j= i_j\quad \text{ for some }1\leq j\leq l ,\\
(-1)^{j-1}\tilde x_{i_1-1}^-\tilde x_{i_2-1}^- \cdots \tilde x_{m+j-1}^-\tilde x_{i_j}^-\cdots \tilde x_{i_l}^-v_\lambda +q\mathcal L(\lambda) & \text{ if }m+j\neq  i_j.
\end{cases}\notag
  \end{align}

Next we have
\begin{align}\label{omegaonbasis}
\tilde\Omega_\psi(k)\tilde x_{i_1}^-\cdots \tilde x_{i_l}^-v_\lambda&=\delta_{k,-i_1} \tilde x_{i_2}\cdots \tilde x_{i_l}^-v_\lambda+\sum_{r\geq 0}g_{q^{-1}}(r)\tilde x^-_{i_1+r}\tilde \Omega_\psi(k-r)\tilde x_{i_2}^-\cdots \tilde x_{i_l}^-v_\lambda  \\
&\equiv \delta_{k,-i_1} \tilde x_{i_2}\cdots \tilde x_{i_l}^-v_\lambda-\tilde x^-_{i_1+1}\tilde \Omega_\psi(k-1)\tilde x_{i_2}^-\cdots \tilde x_{i_l}^-v_\lambda  \mod q\mathcal L(\lambda) \notag \\
&\equiv \delta_{k,-i_1} \tilde x_{i_2}\cdots \tilde x_{i_l}^-v_\lambda-\sum_{j=2}^l(-1)^{j-2}\delta_{k-j+1,-i_j}\tilde x^-_{i_1+1}\tilde x_{i_2+1}^-\cdots   \tilde x_{i_{j-1}+1}^-\tilde x_{i_{j+1}}^-\cdots \tilde x_{i_l}^-v_\lambda  \mod q\mathcal L(\lambda)\notag \\
&\equiv\sum_{j=1}^l(-1)^{j-1}\delta_{k-j+1,-i_j}\tilde x^-_{i_1+1}   \tilde x_{i_2+1}^-\cdots   \tilde x_{i_{j-1}+1}^-\tilde x_{i_{j+1}}^-\cdots \tilde x_{i_l}^-v_\lambda  \mod q\mathcal L(\lambda)\notag
\end{align}
Observe that each summand on the right is ordered so that it is in $\pm \mathcal B
\cup \{0\}$.  The only way in which the whole summation is not in $\pm \mathcal B\cup \{0\}$ is if there are at least two indices $r<s$ such that $i_r=-k+r-1$ and $i_s=-k+s-1$.   But $i_r\geq i_s$ which is a contradiction.  

Condition (v) is satisfied by \eqnref{omegapsi6}.  Indeed we begin by induction.   
Now if $b=\tilde x_{i_1}^-v_\lambda+q\mathcal L(\lambda)$, then we have
\begin{align*}
\tilde \Omega_\psi(k)\tilde x^-_mb&=\delta_{k,-m}b+\sum_{r\geq 0}g_{q^{-1}}(r)\tilde x^-_{m+r}\Omega_\psi(k-r)\tilde x_{i_1}^-v_\lambda  \\
& =\delta_{k,-m}b-\tilde x^-_{m+1}\Omega_\psi(k-1)\tilde x_{i_1}^-v_\lambda  \\
& =\delta_{k,-m}b-\delta_{k-1,-i_1}\tilde x^-_{m+1}v_\lambda .
\end{align*}
Thus if $k=-m$ and $k-1=-i_1$, then $i_1=m+1$ which is a contradiction to the assumption $\tilde x_mb\neq 0$.  Hence $\tilde\Omega_\psi (-m)\tilde x_m^-b=b$. 

Now assuming $\tilde \Omega_\psi(-m)\tilde x_{i_1}^-v_\lambda\equiv \tilde \Omega_\psi(-m)b\neq 0$ by \eqnref{omegaonbasis} we have
\begin{align*}
\tilde \Omega_\psi(-m)\tilde x_{i_1}^-v_\lambda\equiv \delta_{m,i_1} v_\lambda  \mod q\mathcal L(\lambda)\notag
\end{align*}
and we must have $m=i_1$.   Thus $\tilde x_m^-\tilde \Omega_\psi(-m)\tilde x_{i_1}^-v_\lambda\equiv \tilde x_m^-v_\lambda=b$.

 Next take $b=\tilde x_{i_1}^-\cdots \tilde x_{i_k}^-v_\lambda$ with $i_1\geq i_2\geq \cdots\geq i_k$ and if $\tilde x_m^-b\neq 0$, then $i_j\neq m+j$ for all $1\leq j\leq l$ by \eqnref{xonbasis} and we first consider the case $m\geq i_1$: 
 \begin{align*}
\tilde \Omega_\psi(-m)\tilde x_m^-b&\equiv\tilde \Omega_\psi(-m)\tilde x_m^- \tilde x_{i_1}^-\cdots \tilde x_{i_k}^-v_\lambda \\
&\equiv  \tilde x_{i_1}^-\cdots \tilde x_{i_k}^-v_\lambda-  \tilde x_{m+1}^-\tilde \Omega_\psi(-m-1) \tilde x_{i_1}^-\cdots \tilde x_{i_k}^-v_\lambda
 \end{align*}
    If 
$\tilde\Omega_\psi(-m-1)b\neq 0$, then by \eqnref{omegaonbasis} $m+j=i_j$ for some $1\leq j\leq l$ but this contradicts $\tilde x_m^-b\neq 0$.   Hence 
$$
\tilde \Omega_\psi(-m)\tilde x_m^-b\equiv b
$$
for $m\geq i_1$.  
  
  On the other hand assuming $\tilde\Omega_\psi(-m)b\neq 0 $, by \eqnref{omegaonbasis} $m+j-1=i_j$ for some $1\leq j\leq l$, so that 
  \begin{align*}
  \tilde\Omega_\psi(-m)b\equiv (-1)^{j-1}\delta_{-m-j+1,-i_j}\tilde x^-_{i_1+1}   \tilde x_{i_2+1}^-\cdots   \tilde x_{i_{j-1}+1}^-\tilde x_{i_{j+1}}^-\cdots \tilde x_{i_l}^-v_\lambda  \mod q\mathcal L(\lambda)
  \end{align*}
  we have by \eqnref{xonbasis}
  \begin{align*}
\tilde x_m^-  \tilde\Omega_\psi(-m)b&\equiv (-1)^{j-1}\delta_{-m-j+1,-i_j}\tilde x_m^-\tilde x^-_{i_1+1}   \tilde x_{i_2+1}^-\cdots   \tilde x_{i_{j-1}+1}^-\tilde x_{i_{j+1}}^-\cdots \tilde x_{i_l}^-v_\lambda  \mod q\mathcal L(\lambda) \\
&\equiv \delta_{-m-j+1,-i_j} \tilde x^-_{i_1}   \tilde x_{i_2}^-\cdots   \tilde x_{i_{j-1}}^-\tilde x_{m+j-1}^-\tilde x_{i_{j+1}}^-\cdots \tilde x_{i_l}^-v_\lambda  \mod q\mathcal L(\lambda) \\
&\equiv   b.
  \end{align*}
  
  Finally if  $\tilde x_m^-b\neq 0 $, by \eqnref{xonbasis} $m+j\neq i_j$ and 
  \begin{align*}
  \tilde x_m^-b&=(-1)^{j-1}\tilde x_{i_1-1}^-\tilde x_{i_2-1}^- \cdots \tilde x_{m+j-1}^-\tilde x_{i_j}^-\cdots \tilde x_{i_l}^-v_\lambda +q\mathcal L(\lambda) 
  \end{align*}
  so that by \eqnref{omegaonbasis}
    \begin{align*}
  \Omega_\psi(-m)\tilde x_m^-b&=(-1)^{j-1} \Omega_\psi(-m)\tilde x_{i_1-1}^-\tilde x_{i_2-1}^- \cdots \tilde x_{m+j-1}^-\tilde x_{i_j}^-\cdots \tilde x_{i_l}^-v_\lambda +q\mathcal L(\lambda)  \\
  &= \tilde x_{i_1}^-\tilde x_{i_2}^- \cdots \tilde x_{i_j}^-\cdots \tilde x_{i_l}^-v_\lambda +q\mathcal L(\lambda)  \\
  &=b.
  \end{align*}

\end{proof}

\section*{Acknowledgement}
The first two authors would like to thank the Mittag-Leffler Institute for its hospitality during their stay where part of this work was done.  The first author was partially support by a Simons Collaboration Grant (\#319261). The second author was supported in part by the CNPq grant (301320/2013-6) and by the FAPESP grant (2014/09310-5). The third author was partially support by the Simons Foundation Grant \#307555. The authors would like to acknowledge the valuable comments of Masaki Kashiwara on earlier versions of this paper.

\bibliographystyle{alpha}
\bibliography{math}

\def\cprime{$'$}
\providecommand{\bysame}{\leavevmode\hbox to3em{\hrulefill}\thinspace}
\providecommand{\MR}{\relax\ifhmode\unskip\space\fi MR }
\providecommand{\MRhref}[2]{%
  \href{http://www.ams.org/mathscinet-getitem?mr=#1}{#2}
}
\providecommand{\href}[2]{#2}
\bibliographystyle{amsalpha}

\end{document}